\documentclass[11 pt]{amsart}
\usepackage[T1]{fontenc}
\usepackage[utf8]{inputenc}
\usepackage[english]{babel}
\usepackage{mathtools}
\usepackage{amsfonts}
\usepackage{amsthm}
\usepackage{braket}
\usepackage{lipsum}
\newcommand{\h}{\mathrm{H}}
\newcommand{\D}{\mathbb{D}}
\newcommand{\N}{\mathbb{N}}

\newcommand{\C}{\mathbb{C}}

\newcommand{\diag}{\mathrm{diag}}
\newcommand{\de}{\partial}
\newcommand{\id}{\mathrm{Id}}
\newcommand{\spa}{\mathrm{span}}
\newcommand{\dist}{\mathrm{dist}}
\numberwithin{equation}{section}
\theoremstyle{plain}
\newtheorem*{theo*}{Theorem}
\newtheorem{theo}{Theorem}[section]
\newtheorem{coro}[theo]{Corollary}
\newtheorem{prop}[theo]{Proposition}
\newtheorem{lemma}[theo]{Lemma}
\newtheorem{defi}[theo]{Definition}
\theoremstyle{definition}
\newtheorem*{question}{Question}

\newtheorem{rem}[theo]{Remark}
\title{Interpolating Matrices}
\author{Alberto Dayan}
\address{Department of Mathematics\newline Washington University in St. Louis,\newline One Brookings Drive, St. Louis, MO 63130, USA}
\email{alberto.dayan@wustl.edu}
\date{\today}
\begin{document}
\maketitle
\begin{abstract}
We extend Carleson's interpolation Theorem to sequences of matrices, by giving necessary and sufficient separation conditions for a sequence of matrices to be interpolating.
\end{abstract}
\section{introduction}
A sequence $\Lambda=(\lambda_n)_{n\in\N}$ of points in the unit disc $\D$ is \emph{interpolating} for the space $\h^\infty$ of bounded analytic functions on $\D$ if for any bounded sequence $(w_n)_{n\in\N}$ there exists a non zero function $f$ in $\h^\infty$ such that $f(\lambda_n)=w_n$, for any $n$ in $\N$. Intuitively, being interpolating is a matter of how separated the sequence is in $\D$. Since the interpolating functions are holomorphic a natural way to compute distances between points of $\Lambda$ is by using the \emph{pseudo-hyperbolic} distance
\[
\rho(z_1,z_2):=|b_{z_1}(z_2)|=\bigg|\frac{z_1-z_2}{1-\overline{z_1}z_2}\bigg|.
\]
Here $b_\tau$ will denote the involutive Blaschke factor at a point $\tau$ in $\D$. \\
We will say that $\Lambda$  is \emph{strongly separated} if 
\begin{equation}
\label{eqn:ss}
\inf_{n\in\N}\prod_{k\ne n}\rho(\lambda_n, \lambda_k)>0.
\end{equation}
%Observe that $\Lambda$ is weakly separated if and only if there exists $\delta>0$ such that, for any $k\ne n$, there exists a bounded analytic function $f_{k, n}$ such that $||f_{k, n}||_{\infty}=1$, $f_{k, n}(\lambda_n)=0$ and $f_{k, n}(\lambda_k)=\delta$. In the same way, $\Lambda$ is strongly separated if and only if there exists a positive $\delta$ such that, for any $n$ in $\N$, there exists a function $f_n$ so that $||f_n||_{\infty}=1$, $f_n(\lambda_n)=\delta$ and $f_n$ vanishes at all the other points of $\Lambda$.\\

A celebrated result due to Carleson \cite{carloss} affirms that $\Lambda$ is interpolating if and only if it is strongly separated. For a geometric proof, see \cite[Th. 1.1, Ch. VII ]{gar}. In \cite[Th. 9.42]{john}, one can find the same result restated (and proved) in terms of separation conditions on the reproducing kernel functions $(k_{\lambda_n})_{n\in\N}$ in $\h^2$.\\

We use both those approaches to characterize interpolating sequences of matrices (of any dimensions) with spectra in the unit disk.\\
Throughout the discussion, a holomorphic functions $f$ will be applied to a square matrix $M$ via the Riesz-Dunford functional calculus.  Section \ref{sec:intspec} will make it precise, but at this stage just observe that, since we will consider functions holomorphic in the unit disc, the eigenvalues of all the matrices we will consider are in $\D$.\\
A first attempt to emulate the definition of an interpolating sequence of scalars can then be to define a sequence $(A_n)_{n\in\N}$ of square matrices with spectra in $\D$ to be interpolating if, given a sequence $(W_n)_{n\in\N}$ of square matrices which is bounded in the operator norm, then there exists a function $f$ in $\h^\infty$ such that $f(A_n)=W_n$.\\
One reason why this can't be the right approach is given by the following example: let $\Lambda=(\lambda_n)_{n\in\N}$ be a Blaschke sequence in $\D$, and set
\[
A_n=\begin{bmatrix}
\lambda_n & 1\\
0 & \lambda_n
\end{bmatrix}.
\]
The least that we can expect from a definition of interpolating sequence of matrices that is consistent with the scalar case is that $(f(A_n))_{n\in\N}$ is a target sequence too, for any $f$ in $\h^\infty$. But if we set $B$ to be the Blaschke product at the points of $\Lambda$, then $(B'(\lambda_n))_{n\in\N}$ is unbounded, and so is
\[
B(A_n)=\begin{bmatrix}
0 & B'(\lambda_n)\\
0 & 0
\end{bmatrix}.
\]
Therefore, a target sequence for an interpolating sequence of matrices can't simply be a sequence of matrices bounded with respect the operator norm. The problem seems to be that, given a Jordan block
\[
J=\begin{bmatrix}
\lambda & 1\\
0 & \lambda
\end{bmatrix}\qquad\lambda\in\D
\]
and a bounded holomorphic function $f$, then
\[
f(J)=\begin{bmatrix}
f(\lambda) & f'(\lambda)\\
0 & f(\lambda)
\end{bmatrix},
\]
and while $f(\lambda)$ can assume any bounded value, $f'(\lambda)$ has bound that depends on the $\h^\infty$ norm of $f$ and, most importantly, on $\lambda$.\\
The most natural way to overcome this obstruction is to identify a target sequence for a sequence of matrices $(A_n)_{n\in\N}$ with a bounded sequence in $\h^\infty$:
\begin{defi}[Interpolating Matrices]
Let $A=(A_n)_{n\in\N}$ be a sequence of complex square matrices (of perhaps different dimensions) with spectra contained in $\D$. Then $A$ is interpolating for $\h^\infty$ if, for any bounded sequence $(\phi_n)_{n\in\N}$ in $\h^\infty$ there exists a function $\varphi$ in $\h^\infty$ such that 
\[
\varphi(A_n)=\phi_n(A_n)\qquad n\in\N.
\]
\end{defi}
Observe that if the matrices in $A$ reduce to scalars, then we recover the classic definition of an interpolating sequence in $\D$.\\

Also the notion of strong separation can be generalized to matrices. Let $M$ be a square matrix, and $\lambda$ be one of its eigenvalues. We define the \emph{order} of $\lambda$ as the maximum of the sizes of all the Jordan blocks of $M$ associated to $\lambda$. If $M$ has eigenvalues $\lambda_1, \dots, \lambda_k$ with orders $m_1, \dots, m_k$ respectively, define
\[
B_M(z):=\prod_{j=1}^kb_{\lambda_j}^{m_j}(z)\qquad z\in\D
\]
as the \emph{Blaschke factor associate to $M$}. Observe that $B_M$ is an inner function and that $B_M(M)=0$. In particular, $B_M$ is the product of the least number of Blaschke factors at the eigenvalues of $M$ that vanishes at $M$. \\
One can also define $B_M$ from the minimal polynomial $P_M$ of $M$:  $B_M=P_M/\widetilde{P}_M$,  where
\[
\widetilde{p}(z):=z^{\deg(p)}\overline{p}\left(\frac{1}{\overline{z}} \right),\qquad z\in\D.
\]
\begin{defi}
\label{defi:sep}
A sequence of square matrices $(A_n)_{n\in\N}$ is strongly separated if
\begin{equation}
\label{eqn:ssm}
\inf_{z\in\D}\sup_{n\in\N}\prod_{k\ne n}|B_{A_k}(z)|>0.
\end{equation}
\end{defi}
Our first result can be stated as follows:
\begin{theo}
\label{theo:intmat}
Let $A=(A_n)_{n\in\N}$ be a sequence of matrices with spectra in $\D$. Then $A$ is interpolating if and only if it is strongly separated.
%\begin{equation}
%\label{eqn:bcondition}
%\sum_{n\in\N}\sum_{j=1}^{k_n}m_{n, j}(1-|\lambda_{n, j}|)<\infty
%\end{equation}
%and
\end{theo}
In section \ref{sec:sep} we will see how condition \eqref{eqn:ssm} relates to \eqref{eqn:ss}, and we will prove that a sequence of matrices satisfying \eqref{eqn:ssm} actually exists, for any sequence of dimensions.\\
 Section \ref{sec:intspec} will then translate the interpolation problem for matrices to a generalized interpolation problem for their spectra, and give a proof of Theorem \ref{theo:intmat}.\\
In Section \ref{sec:beu}, we will adapt the argument in \cite[Ch. VII, Th. 2.2]{gar} to construct the so called \emph{P.Beurling functions} associated to a sequence of interpolating matrices.\\
 Section \ref{sec:kernels} will deal with a different approach to separation in the unit disk by looking at model spaces in $\h^2$:  let $(A_n)_{n\in\N}$ be a sequence of matrices, with associated Blaschke factors
\[
B_{A_n}=\prod_{j=1}^{k_n}b_{\lambda_{n, j}}^{m_{n, j}}\qquad n\in\N.
\]
Define for any $n$ in $\N$ the {\bf model space} of $A_n$ as
\begin{equation}
\label{eqn:model}
H_n:=\h^2\ominus B_{A_n}\h^2.
\end{equation}
This will allow us to give another characterization for a sequence of matrices to be interpolating. In particular, one can make sense of a notion of separation for subspaces of a Hilbert space: if $L$ and $K$ are two closed subspaces of a Hilbert space $\mathcal{H}$, the sine between $L$ and $K$ is 
\[
\sin(L, K):=\inf\{\dist(k, L)\,\, | \,\, k\in K, ||k||=1\}.
\]
Namely, $\sin(L, K)$ is the smallest choice for the sine of an angle between a vector of $L$ and a vector of $K$.
%A sequence $(K_n)_{n\in\N}$ of subspaces of a Hilbert space $\mathcal{H}$ is \emph{weakly separated} if 
%\[
%\inf_{j\ne n}\dist(H_j, H_n)>0,
%\]

\begin{defi}
\label{defi:sep:model}
A sequence $(K_n)_{n\in\N}$ of subspaces of a Hilbert space  is strongly separated if 
\begin{equation}
\label{eqn:sep:model}
\inf_{n\in\N}\sin\left(K_n\,\, ,\,\, \underset{j\ne n}{\overline{\spa}}\{K_j\}\right)>0;
\end{equation}
\end{defi}
Theorem \ref{theo:dist:model} will show that such a separation condition for the model spaces of $(A_n)_{n\in\N}$ coincide with the one given in \eqref{eqn:ssm}. Therefore,
\begin{theo}
A sequence $(A_n)_{n\in\N}$ of matrices with spectra in the unit disk is interpolating if and only if their model spaces are strongly separated.
\end{theo}
To conclude, in Section \ref{sec:bessel} we ask if the well known characterization of interpolating sequences given in term of Carleson measures extends to sequences of matrices.\\

The author would like to thank John M$^{\text{c}}$Carthy for the constant and valuable help during the development of this work. 
\section{Separation in the Unit Disc}
\label{sec:sep}
Let $(\lambda_n)_{n\in\N}$ be a sequence in $\D$. Note that no such sequence is strongly (or weakly) separated if some terms are repeated. If we want to take into account multiplicities, a natural generalization for strong separation is
\begin{equation}
\label{eqn:ssmult}
\inf_{n\in\N}\prod_{k\ne n}\rho(\lambda_n,\lambda_k)^{m_k}>0,
\end{equation}
where $m_k$ is the multiplicity of $\lambda_k$. A more demanding generalization of \eqref{eqn:ss} is due to Nikolski \cite[Th. 3.2.14]{nik} and yields  a uniform lower bound for such a product that has to be checked for all the points of $\D$:
\begin{defi}[Uniform Strong Separation]
Let $\Lambda=(\lambda_n)_{n\in\N}$ be a sequence in $\D$, with a given sequence $(m_n)_{n\in\N}$ of  multiplicities. Then $\Lambda$ is uniformly strongly separated if
\begin{equation}
\label{eqn:ussm}
\inf_{z\in\D}\sup_{n\in\N}\prod_{k\ne n}\rho(z,\lambda_k)^{m_k}>0.
\end{equation}
\end{defi}
It can be easily checked that inequality \eqref{eqn:ssmult} is inequality \eqref{eqn:ussm} for $z$ in $\Lambda$. Therefore, a uniformly strongly separated sequence is strongly separated. Before asking ourselves if, or under which conditions, a strongly separated sequence is uniformly strongly separated, we should check that such sequences exist:
\begin{theo}
\label{theo:exuss}
For any sequence of multiplicities $(m_n)_{n\in\N}$ and for any $0<\delta<1$, there exists a sequence $(\lambda_n)_{n\in\N}$ such that
\[
\inf_{z\in\D}\sup_{n\in\N}\prod_{k\ne n}\rho(z,\lambda_k)^{m_k}\geq\delta.
\] 
\end{theo}
\begin{proof}
Fix $\delta<1$ and $1<\nu<1/\delta$. Set $B_k:=b_{\lambda_k}^{m_k}$. Observe that
\[
\rho(z,\lambda_k)^{m_k}=|B_k(z)|,\qquad k\in\N.
\]
The proof now proceeds inductively:
\begin{itemize}
\item \emph{There exist $\lambda_1$ and $\lambda_2$ in $\D$ so that
\[
\inf_{z\in\D}\max_{k=1,2}|B_k(z)|\geq\delta\nu.
\]}
To see that it is true, choose any $\lambda_1$ in $\D$, and then a radius $r<1$ such that 
\[
|B_1(z)|\geq\delta\nu,\qquad z\in\{ r\leq|z|\leq 1\}.
\]
It suffices then to choose $\lambda_2$ close enough to $\de\D$ so that
\[
|B_2(z)|\geq\delta\nu\qquad |z|\leq r,
\]
by observing that the family $(b_\tau)_{\tau\in\D}$ is locally uniformly bounded, and therefore by Montel's Theorem 
\[
\lim_{n\to\infty}|b_{\tau_n}|=1
\]
uniformly on compact subsets of $\D$, for some sequence $(\tau_n)_{n\in\N}$ approaching the unit circle.
\item \emph{If $\lambda_1,\dots,\lambda_n$ are points in $\D$ so that
\[
\inf_{z\in\D}\max_{l=1,\dots, n}\prod_{k\ne l}|B_k(z)|\geq\gamma>0,
\]
then there exists a point $\lambda_{n+1}$ such that
\[
\inf_{z\in\D}\max_{l=1,\dots, n+1}\prod_{k\ne l}|B_k(z)|\geq\frac{\gamma}{\nu^{2^{-n}}}.
\]}
To prove it, observe that
\[
\inf_{z\in\D}\max_{l=1,\dots, n+1}\prod_{k\ne l}|B_k(z)|=\inf_{z\in\D}\max\bigg\{\prod_{k=1}^n|B_k(z)|\,;\,|B_{n+1}(z)|\max_{l=1,\dots, n}\prod_{k\ne l} |B_k(z)| \bigg\} .
\]
Then it suffices to find a radius $r_n$ so that
\[
\prod_{k=1}^n|B_k(z)|\geq\gamma\qquad z\in\{r_n\leq|z|<1\}
\]
and $\lambda_{n+1}$ so that 
\[
|B_{n+1}(z)|\geq\nu^{-2^{-n}}\qquad |z|\leq r_n.
\]
\end{itemize}
This inductive construction shows that the sequence $(\lambda_{n})_{n\in\N}$ that one can build in this fashion satisfies
\[
\inf_{z\in\D}\max_{l=1,\dots,n}\prod_{k\leq n, k\ne l}|B_k(z)|\geq\delta\nu^{1-1/2-\dots-1/2^n},\qquad n\in\N.
\]
To conclude, let $n$ going to $\infty$.
\end{proof}
\begin{rem}
One can note that in the above proof we used only the fact that each $B_k$ is a finite Blaschke product. Therefore the same argument can be applied to the sequence of inner functions
\[
B_k:=\prod_{j=1}^{n_k}b_{\lambda_{k, j}}^{m_{k, j}},
\]
for any sequence $(n_{k})_{k\in\N}$ and $\{m_{k, j}\,\,|\,\, j=1,\dots, n_k\}_{k\in\N}$. This gives that a sequence of matrices $(A_n)_{n\in\N}$ satisfying \eqref{eqn:ssm} exists, for any choice that we can make of their dimensions.
\end{rem}
Nikolski gave in [Th. 3.2.14.]{nik} an analogue of \eqref{eqn:ussm} for any convergent product of inner functions, and related it to a generalized interpolation problem:
\begin{theo}
\label{theo:nik}
Let $(\Theta_n)_{n\in\N}$ be a sequence of inner functions in $\D$ so that $\Theta:=\prod_{n\in\N}\Theta_n$ is an inner function. Then the following are equivalent:
\begin{description}
\item[(i)]There exists a positive $\delta$ such that
\[
\inf_{z\in\D}\sup_{n\in\N}\prod_{k\ne n}|\Theta_k(z)|\geq\delta;
\]
\item[(ii)] For any bounded sequence $(\phi_n)_{n\in\N}$ in $\h^\infty$ there exists a $\varphi$ in $\h^\infty$ such that
\begin{equation}
\label{eqn:genint}
\varphi-\phi_n\in\Theta_n\h^\infty,\qquad n\in\N.
\end{equation}
\end{description}
\end{theo}
Observe that, if $(\lambda_{n})_{n\in\N}$ is a sequence in $\D$ and $\Theta_n=b_{\lambda_n}^{m_n}$, then condition (ii) of Theorem \ref{theo:nik} is an interpolating property for $\h^\infty$ including derivatives. Namely, (ii) holds if and only if for any bounded sequence $(\phi_n)_{n\in\N}$ in $\h^\infty$ there exists a bounded analytic function $\varphi$ so that
\[
\varphi^{(j)}(\lambda_n)=\phi_n^{(j)}(\lambda_n),
\]
for any $n$ in $\N$ and $j=0,\dots, m_n-1$. \\
It is not difficult to show that, if the sequence of multiplicities $(m_n)_{n\in\N}$ is bounded, then any strongly separated sequence is actually uniformly strongly separated: a proof can be found in \cite[Lemma 3.2.18.]{nik}. Therefore, Theorem \ref{theo:nik} generalizes Carleson's interpolation Theorem for those interpolation problems in $\h^\infty$ that aim to specify $m_n$ derivatives at each $\lambda_n$.\\
On the other hand, if $(m_n)_{n\in\N}$ is unbounded, a strongly separated sequence is not in general uniformly strongly separated:
\begin{theo}
\label{theo:notunif}
Let $(m_n)_{n\in\N}$ be an increasing and unbounded sequence. Then there exists a sequence $\Lambda=(\lambda_n)_{n\in\N}$ in $\D$ whose sequence of multiplicities is $(m_n)_{n\in\N}$ and so that $\Lambda$ is strongly separated but not uniformly strongly separated.
\end{theo}
The proof has a strong geometric flavor: the first tool that is going to be used is the following property of the pseudo-hyperbolic metric:
\begin{lemma}
\label{lemma:rho}
Let $0<\lambda_1<\gamma<\lambda_2<1$, and suppose that $\rho(\gamma,\lambda_1)=t\rho(\lambda_1,\lambda_2)$, for some $0<t<1$. Then
\[
\rho(\gamma,\lambda_2)=\frac{1-t}{1-\rho^2(\lambda_1,\lambda_2)t}\rho(\lambda_1,\lambda_2).
\]
\end{lemma} 
\begin{proof}
Let $s_i=\rho(\lambda_i, \gamma)$, $i=1,2$. Then $s_1=-b_{\lambda_1}(\gamma)$, and therefore
\[
\gamma=b_{\lambda_1}(-s_1)=\frac{\lambda_1+s_1}{1+\lambda_1s_1},
\]
since Blaschke factors are their own inverses. Therefore 
\[
s_2=b_{\lambda_2}(\gamma)=b_{\lambda_2}(b_{\lambda_1}(-s_1))=\frac{\rho(\lambda_1,\lambda_2)-s_1}{1-\rho(\lambda_1,\lambda_2)s_1}.
\]
This, together with
\[
s_1=t\rho(\lambda_1, \lambda_2),
\]
concludes the proof.
\end{proof}
Of course, both \eqref{eqn:ssmult} and \eqref{eqn:ussm} are separation conditions for the sequence $\Lambda=(\lambda_n)_{n\in\N}$. Rather than proving that a uniform strongly separated sequence is more separated than a strongly separated one, the construction that will prove Theorem \ref{theo:notunif} is given by a sequence $\Lambda$ that lacks regularity:
\begin{lemma}
\label{lemma:nu}
Let $(m_n)_{n\in\N}$ be a sequence of multiplicities, and let $0<\nu<1$. Then there exists a strongly separated sequence $(\lambda_n)_{n\in\N}$ in $\D$ such that 
\begin{equation}
\label{eqn:nu}
\rho(\lambda_n,\lambda_{n+1})^{m_{n+1}}=\nu
\end{equation}
infinitely often.
\end{lemma}
\begin{proof}
Choose a positive $\lambda_1$, and then $\lambda_2>\lambda_1$ so that 
\[
\rho(\lambda_1,\lambda_2)^{m_2}=\nu.
\]
Then pick $\lambda_3>\lambda_2$ which is close enough to $1$ so that
\[
\begin{cases}
&\rho(\lambda_3,\lambda_1)^{m_1}\rho(\lambda_3,\lambda_2)^{m_2}\geq\frac{1}{2}\\
&\rho(\lambda_2,\lambda_3)^{m_3}\rho(\lambda_2,\lambda_1)^{m_1}\geq\nu\rho(\lambda_2,\lambda_3)^{m_3}\geq\nu~ 2^{-\frac{1}{2^3}}\\
&\rho(\lambda_1,\lambda_2)^{m_2}\rho(\lambda_1,\lambda_3)^{m_3}=\nu\rho(\lambda_1, \lambda_3)^{m_3}\geq \nu~2^{-\frac{1}{2^3}}
\end{cases},
\]
and $\lambda_4>\lambda_3$ so that $\rho(\lambda_3,\lambda_4)=\nu$. In this fashion, the choices of the even terms are going to be at distance $\nu$ from the previous odd, while an odd term $\lambda_{2k+1}$ will be chosen so that 
\begin{equation}
\label{eqn:partialss}
\begin{cases}
&\prod_{j=1}^{2k}\rho(\lambda_{2k+1},\lambda_j)^{m_j}\geq\frac{1}{2}\\
&\rho(\lambda_{2k+1},\lambda_l)^{m_{2k+1}}\prod_{j\leq 2k, j\ne l}\rho(\lambda_{l}, \lambda_j)^{m_j}\geq\nu~ 2^{\sum_{j=1}^k2^{-(2k+1)}}\qquad l=1,\dots, 2k.
\end{cases}
\end{equation}
We conclude the proof by letting $k$ going to $\infty$.
\end{proof}
We are now ready to prove Theorem \ref{theo:notunif}:
\begin{proof}[Proof of Theorem \ref{theo:notunif}]
 For any $n$ in $\N$, choose $\lambda_{2n-1}<\xi_n<\lambda_{2n}$ so that $\rho(\xi_n,\lambda_{2n-1})=t_n\rho(\lambda_{2n-1},\lambda_{2n})$, where 
\begin{equation}
\label{eqn:tn}
\lim_{n\to\infty}t_n=1
\end{equation}
but
\begin{equation}
\label{eqn:t_n}
\lim_{n\to\infty}t_n^{m_{2n-1}}=0.
\end{equation}
Observe that we can choose such a $(t_n)_{n\in\N}$ since $m_n$ goes to infinity. To prove the Theorem, it suffices to show that, if $\rho(\xi_n,\lambda_{2n})=s_n\rho(\lambda_{2n-1},\lambda_{2n})$, then
\begin{equation}
\label{eqn:s_n}
\lim_{n\to\infty}s_n^{m_{2n}}=0
\end{equation}
too. Indeed, if both equation \eqref{eqn:t_n} and equation \eqref{eqn:s_n} hold, $\xi_n$ is close enough to both $\lambda_{2n-1}$ and $\lambda_{2n}$ so that
\[
\rho(\xi_n, \lambda_{2n-i})^{m_{2n-i}}\underset{n\to\infty}{\to}0\qquad i=0, 1,
\]
and therefore \eqref{eqn:ussm} fails.\\
 Thanks to Lemma \ref{lemma:rho}, this is equivalent to showing that
\[
\lim_{n\to\infty}\bigg(\frac{1-t_n}{1-\rho^2(\lambda_{2n-1},\lambda_{2n})t_n}\bigg)^{m_{2n}}=\lim_{n\to\infty}\bigg(1-\frac{t_n(1-\rho^2(\lambda_{2n-1},\lambda_{2n}))}{1-\rho^2(\lambda_{2n-1}, \lambda_{2n})t_n}\bigg)^{m_{2n}}=0,
\]
which is true if and only if
\[
\lim_{n\to\infty}\frac{m_{2n}t_n(1-\rho^2(\lambda_{2n-1},\lambda_{2n}))}{1-\rho^2(\lambda_{2n-1},\lambda_{2n})t_n}=\infty.
\]
Set $\gamma_n:=1-t_n$ and $\varepsilon_n:=1-\rho^2(\lambda_{2n-1}, \lambda_{2n})$. Recall that by construction $\rho(\lambda_{2n-1},\lambda_{2n})^{m_{2n-1}}\geq\nu>0$, and therefore by \eqref{eqn:t_n} we have
\[
\gamma_n+\varepsilon_n\underset{n\to\infty}{\sim}\gamma_n,
\] 
and so 
\[
\frac{m_{2n}t_n(1-\rho^2(\lambda_{2n-1},\lambda_{2n}))}{1-\rho^2(\lambda_{2n-1},\lambda_{2n})t_n}=\frac{m_{2n}(1-\gamma_n)\varepsilon_n}{1-(1-\gamma_n)(1-\varepsilon_n)}\underset{n\to\infty}{\sim} m_{2n}\frac{\varepsilon_n}{\gamma_n}.
\]
To conclude observe that, thanks to \eqref{eqn:nu} and \eqref{eqn:tn},  
\[
m_{2n}\frac{\varepsilon_n}{\gamma_n}=m_{2n}\frac{1-\nu^{\frac{2}{m_{2n}}}}{\gamma_n}\underset{n\to\infty}{\sim}\frac{-2\log\nu}{\gamma_n}\to\infty.
\]
\end{proof}
This exhibits a reason why the notion of strongly separated matrices in \eqref{eqn:ssm} involves a uniform lower bound on all the point on $\D$, rather than a condition that has to be checked only at the eigenvalues of the matrices composing the sequence. Even though the latter condition would be more similar to \eqref{eqn:ss}, Theorem \ref{theo:notunif} shows that it doesn't extend to the matrix case when the orders of the eigenvalues are unbounded. 
\begin{rem}
A characterization of uniform strong separation given only in terms of the mutual distance of the points of a the sequence $(\lambda_n)_{n\in\N}$ can be found in \cite[Cor. 5, Lec. X]{nik2}. Indeed, for any sequence of multiplicities $(m_n)_{n\in\N}$, $(\lambda_n)_{n\in\N}$ is uniformly strongly separated if and only if 
\[
\inf_{n\in\N}\prod_{k\ne n}\rho(\lambda_n, \lambda_k)^{m_nm_k}>0,
\] 
provided that $(\lambda_{n\in\N})$ approaches the unit circle non tangentially.
\end{rem}
\section{Interpolating Spectra}
\label{sec:intspec} 
Let $M$ be a square matrix of dimension $m$ with eigenvalues in $\D$.  Pick a closed curve $\gamma$ having winding number $1$ around the spectrum of $M$ and contained in $\D$. Then we can extend the Cauchy integral formula to define, for any $f$ holomorphic in $\D$, the $m\times m$ matrix $f(M)$ by
\[
f(M):=\frac{1}{2\pi i}\int_\gamma f(w)(w\cdot\id_m-M)^{-1}\,dw.
\]
Another way to define $f(M)$ arises from the power series expansion of $f(z)=\sum_{n=0}^\infty a_nz^n$:
\[
f(M):=\sum_{n=0}^\infty a_n M^n.
\]
Under our assumptions on $M$, those two methods give as output the same $f(M)$. For instance, see \cite[Ch. VII]{ds}.
Observe that for any $P$ invertible of size $m$ and for any function $f(z)=\sum_{n=0}^\infty a_nz^n$ holomorphic in $\D$, 
\[
f(PMP^{-1})=\sum_{n=0}^\infty a_n(PMP^{-1})^n=P\bigg(\sum_{n=0}^\infty a_n M^n\bigg)P^{-1}=Pf(M)P^{-1}.
\]
Therefore a sequence of matrices $A=(A_n)_{n\in\N}$ is interpolating if and only if the sequence $A'=(P_nA_nP_n^{-1})_{n\in\N}$ is interpolating, for any sequence $(P_n)_{n\in\N}$ of invertible matrices of the right sizes.\\
For us $A'$ will be the sequence of the Jordan canonical forms of the matrices in $A$. Observe also that for any Jordan block of the form
\[
J=\begin{bmatrix}
\lambda  &1 &0  &0\\
0           &\ddots &\ddots  &\vdots\\
\vdots &\vdots &\ddots &1\\
0 &0 &\hdots  &\lambda
\end{bmatrix}
\]  
of size $m$ we have
\[
f(J)=\begin{bmatrix}
f(\lambda)  &f'(\lambda) &\dots  &\frac{f^{(m-1)}(\lambda)}{(m-1)!}\\
0           &\ddots &\ddots  &\vdots\\
\vdots &\vdots &\ddots &f'(\lambda)\\
0 &0 &\hdots  &f(\lambda)
\end{bmatrix},
\]
for any $\lambda$ in $\D$ and $f$ holomorphic in the unit disk. Thus, to interpolate a sequence of matrices one has to solve a generalized interpolating problem on the eigenvalues of those matrices. Indeed, to any matrix $A_n$ of $A$ we associate the list
\[
(\lambda_{n, j}, m_{n, j})_{j=1}^{k_n}
\]
of its eigenvalues $\lambda_{n, j}$ with associated orders $m_{n, j}$. Namely, any $(\lambda_{n, j}, m_{n, j})$ corresponds to a Jordan block $J_{n, j}$. Thus, without loss of generality
\begin{equation}
\label{eqn:diag}
f(A_n)=\diag(f(J_{n, 1}),\dots, f(J_{n, k_n})),
\end{equation}
where
\begin{equation}
\label{eqn:jblock}
f(J_{n, j})=\begin{bmatrix}
f(\lambda_{n, j})  &f'(\lambda_{n, j}) &\dots  &\frac{f^{(m_{n, j}-1)}(\lambda_{n, j})}{(m_{n, j}-1)!}\\
0           &\ddots &\ddots  &\vdots\\
\vdots &\vdots &\ddots &f'(\lambda_{n, j})\\
0 &0 &\hdots  &f(\lambda_{n, j})
\end{bmatrix}.
\end{equation}
\\
Indeed, for interpolating purposes only, a Jordan block of $A_n$ associated to one of its eigenvalues which doesn't have maximal size is irrelevant. Therefore, \eqref{eqn:diag} and \eqref{eqn:jblock} translate an interpolating problem for matrices to condition \eqref{eqn:genint}, once we set $\Theta_n=B_{A_n}$. Therefore, Theorem \ref{theo:intmat} follows from Theorem \ref{theo:nik}.
\section{P. Beurling Functions}
\label{sec:beu}
Let $(\lambda_n)_{n\in\N}$ be a sequence in $\D$. An elegant way to solve any interpolation problem with bounded nodes $(w_n)_{n\in\N}$ is to construct a sequence of functions $(f_n)_{n\in\N}$ in $\h^\infty$ such that $f_n(\lambda_k)=\delta_{n, k}$, and so that
\begin{equation}
\label{eqn:beu}
\sup_{z\in\D}\sum_{n\in\N}|f_n(z)|<\infty.
\end{equation}
Given such a sequence $(f_n)_{n\in\N}$, one can indeed observe that $f:=\sum_{n\in\N}w_nf_n$ is a bounded analytic function that interpolates $(\lambda_n)_{n\in\N}$ with $(w_n)_{n\in\N}$.\\
The functions $(f_n)_{n\in\N}$ are usually called {\bf P. Beurling functions} associated to $(\lambda_n)_{n\in\N}$. In \cite[Ch. VII, Th. 2.2]{gar}, one can find a proof of the existence of the P. Beurling functions for any interpolating sequence on a uniform algebra of holomorphic function on a compact set.\\
 It turns out that the same proof adapts to the case of a sequence $(A_n)_{n\in\N}$ of interpolating matrices:
\begin{theo}
\label{theo:beu}
Let $(A_n)_{n\in\N}$ be a sequence of interpolating matrices. Then there exists a sequence $(f_n)_{n\in\N}$ in $\h^\infty$ so that \eqref{eqn:beu} holds and $f_n(A_k)=\delta_{n, k}\cdot\mathrm{Id}$.
\end{theo}
As for the scalar case, Theorem \ref{theo:beu} gives an algorithm to interpolate any bounded target sequence $(\phi_n)_{n\in\N}$ in $\h^\infty$: it suffices to set
\[
f:=\sum_{n\in\N}f_n\phi_n\in\h^\infty,
\]
and therefore  $f(A_n)=\phi_n(A_n)$, for any $n$.\\

Before proving Theorem \ref{theo:beu}, we define the \emph{constant of interpolation } for a sequence $A=(A_n)_{n\in\N}$ of interpolating matrices. Consider the Banach space
\[
\mathcal{L}^\infty:=\left\{(f_n)_{n\in\N}\subset\h^\infty\,\,|\,\, \sup_{n\in\N}||f_n||_\infty<\infty\right\},
\]
and define the equivalence relation $\sim_A$ on $\mathcal{L}^\infty$ so that two sequences are equivalent if and only if they agree on $A$, that is,
\[
(f_n)_{n\in\N}\sim_A(g_n)_{n\in\N}\quad\iff\quad f_n(A_n)=g_n(A_n),\qquad n\in\N.
\]
Set $\mathcal{L}^\infty_A:=\mathcal{L}^\infty/\sim_A$, and observe that the linear contraction
\[
E_A: h\in\h^\infty\mapsto[(h)_{n\in\N}]_{\sim_A}\in\mathcal{L}^\infty_A
\]
is surjective, since $A$ is interpolating. The open map Theorem implies that there exists a constant $M$ so that, if $\sup_{n\in\N}||\phi_n||_\infty\leq1$, and $f(A_n)=\phi_n(A_n)$ for any $n$, one has 
\[
||f||_\infty\leq M.
\]
The least $M$ that works is denoted by interpolation constant of the sequence $A$.\\

Theorem \ref{theo:beu} now follows from a normal family argument and the following Proposition:
\begin{prop}
Let $A$ be an interpolating sequence of matrices, and $M$ be its interpolation constant. Then, for any $n$ in $\N$ and for any $\varepsilon>0$, there exist $f_1$, $f_2$, \dots, $f_n$ in $\h^\infty$ so that 
\[
f_i(A_j)=\delta_{i, j}\cdot\mathrm{Id}\qquad i,j=1, \dots, n
\]
and 
\[
\sup_{z\in\D}\sum_{j=1}^n|f_j(z)|\leq M^2+\varepsilon.
\]
\end{prop}
The proof is a straightforward adaptation of the elegant argument in \cite[h. VII, Th. 2.2]{gar}:
\begin{proof}
Fix a positive integer $n$, a positive $\varepsilon$, and let $\omega$ be a primitive $n$-th root of unity in $\C$. For any $j=1, \dots, n$, solve the the interpolating problem with constant data
\[
\phi_k=\begin{cases}
\omega^{jk}\quad&\text{if}\, 1\leq k\leq n\\
0\quad&\text{if}\, k>n
\end{cases},
\]
and find then, for any positive $\gamma$, a function $g_j$ so that $||g_j||_\infty\leq M+\gamma$ and 
\[
g_j(A_k)=\omega^{jk}~\mathrm{Id},\qquad j,k=1, \dots, n.
\]
Define
\[
f_j(z):=\left(\frac{1}{n}\sum_{k=1}^n\omega^{-jk}g_k(z)\right)^2,
\]
and since $\omega$ is a primitive $n$-th root of unity one can check that $f_j(A_k)=\delta_{j, k}\cdot\mathrm{Id}$. Moreover, for any $z$ in $\D$,
\[
\begin{split}
\sum_{j=1}^n|f_j(z)|&=\frac{1}{n^2}\sum_{j=1}^n\left(\sum_{l,k=1}^n \omega^{(l-k)j}g_k(z)\overline{g_l(z)}\right)\\
&=\frac{1}{n^2}\sum_{k, l=1}^ng_k(z)\overline{g_l(z)}\sum_{j=1}^n\omega^{(l-k)j}\\
&=\frac{1}{n^2}\sum_{k=1}^nn|g_k(z)|^2\leq(M+\gamma)^2\leq M^2+\varepsilon,
\end{split}
\]
if $\gamma$ is chosen small enough.
\end{proof}

In \cite{beu}, Vinogradov, Gorin and Khruschen deduced the existence of the P. Beurling functions for a sequence $\Lambda=(\lambda_n)_{n\in\N}$ in $\D$ by only assuming that $\Lambda$ is strongly separated. It would be interesting for us to find a similar argument for a matrix interpolation problem.
\section{Separation via Model Spaces}
\label{sec:kernels}
In this section, we will see how a separation condition like \eqref{eqn:ssm} can be seen as an almost orthogonality property for kernel functions in the Hardy space $\h^2=\h^2(\D)$. Let $(z_n)_{n\in\N}$ be a sequence in the unit disc, and $(m_n)_{n\in\N}$ be a sequence of positive integers. As we saw in Section \ref{sec:sep}, to study separation properties of the sequence $(z_n)_{n\in\N}$ we need to form the Blaschke product
\[
B:=\prod_{n\in\N}b_{z_n}^{m_n}.
\]
For any $w$ in $\D$, denote by $k_w$ the reproducing kernel of $w$ in $\h^2$, and let $k^{(j)}_w=\partial^jk_w/\partial\overline{w}^j$ be the reproducing kernel for the $j$-th derivatives at $w$, that is,
\[
\braket{f, k^{(j)}_w}=f^{(j)}(w)\qquad f\in\h^2.
\]
Observe that a kernel function $k_w$ belongs to 
\[
\h^2\ominus B\h^2=\spa\{k_{z_n}, k^{(1)}_{z_n},\dots,k^{(m_n-1)}_{z_n} | n\in\N\}
\]
 if and only if $B(w)=0$. Indeed, the only zeros of $B$ are in the sequence $(z_n)_{n\in\N}$, and therefore if $B(w)=0$ then $z$ must actually be a point in the sequence. Conversely, if $k_w$ belongs to the subspace of $\h^2$ spanned by $\{k_{z_n}, k^{(1)}_{z_n},\dots,k^{(m_n-1)}_{z_n} | n\in\N\}$, then we can write
\[
k_w=\sum_{i=1}^l\sum_{j=0}^{m_{n_i}-1}a_{i, j}k^{(j)}_{z_{n_i}},
\]
and
\[
B(w)=\braket{B, k_z}_{\h^2}=\sum_{i=1}^l\sum_{j=0}^{m_{n_i}-1}a_{i, j}\braket{B, k^{(j)}_{z_{n_i}}}_{\h^2}=0,
\]
since $B$ has a zero of multiplicity $m_n-1$ at each $z_n$.\\
More generally, 
\begin{theo}
\label{theo:dkernels}
For any $d\geq1$ and for any $\Theta$  inner function in $\h^2(\D^d)$
\begin{equation}
\label{eqn:dist:span}
\dist\bigg(\frac{k_w}{||k_w||}\,\,,\,\,\h^2(\D^d)\ominus\Theta\h^2(\D^d)\bigg)=|\Theta(w)|.
\end{equation}
\end{theo}
In particular, if $d=1$ and $\Theta=B$, \eqref{eqn:dist:span} says that $|B(w)|$ is the distance between the normalized kernel function at $w$ and  $\spa\{k_{z_n}, k^{(1)}_{z_n},\dots,k^{(m_n-1)}_{z_n} | n\in\N\}$.\\

 Now, let $(A_n)_{n\in\N}$ be a sequence of matrices with spectra in the unit disk and let $(H_n)_{n\in\N}$ be its associated sequence of model spaces in $\h^2$, as in \eqref{eqn:model}.  Theorem \ref{theo:intmat} and Theorem \ref{theo:dkernels} imply immediately that the sequence $(A_n)_{n\in\N}$ is interpolating if and only if  $(H_n)_{n\in\N}$ is, in a certain sense, well separated:
\begin{theo}
A sequence of matrices $(A_n)_{n\in\N}$ is interpolating if and only if 
\begin{equation}
\label{eqn:sepker}
\inf_{z\in\D}\sup_{n\in\N}\dist\bigg(\frac{k_z}{||k_z||}\,\,,\,\, \underset{k\ne n}{\spa}\{H_k\}\bigg)>0.
\end{equation}
\end{theo}
\begin{proof}
Thanks to Theorem \ref{theo:intmat}, $(A_n)_{n\in\N}$ is interpolating if and only if it is strongly separated, that is if and only if
\[
\inf_{z\in\D}\sup_{n\in\N}\prod_{k\ne n}|B_{A_k}(z)|>0.
\]
Since, for any $n$, $\underset{k\ne n}{\overline{\spa}}\{H_k\}=\h^2\ominus \left(\prod_{k\ne n}B_{A_k}\right)\h^2$, equation \eqref{eqn:dist:span} concludes the proof.
\end{proof}
The proof of Theorem \ref{theo:dkernels} follows from two Lemmas. The first one states a geometric property of any Hilbert space:
\begin{lemma}
\label{lemma:dist}
Let $K$ be a closed subspace of a Hilbert space $\mathcal{H}$, and $x$ be a point of $\mathcal{H}\setminus K$. Then
\[
\frac{1}{\dist(x, K)}=\inf\{||y|| \,\,|\,\, y\in K^\perp\,\, ,\,\,\braket{y, x}=1\}.
\]
\end{lemma}
For a proof, see \cite[Prop. 12.6]{ross}. The second Lemma is the solution of a very natural extremal problem for reproducing kernel Hilbert spaces. Given a reproducing kernel Hilbert space $\mathcal{H}$ on a set $X$, let $\braket{\cdot,\cdot}_{\mathcal{H}}$, $||\cdot||_{\mathcal{H}}$ and $k^\mathcal{H}_z$ denote the inner product, the norm and a reproducing kernel at $z$ in $X$, respectively.
\begin{lemma}
\label{lemma:h2}
Let $a$ be a point of $\C$, and $z$ in $X$. Then
\[
\inf\{||g||_\mathcal{H}\,\,|\,\,g\in\mathcal{H}, \,\,g(z)=a \}=\frac{|a|}{||k^\mathcal{H}_z||_{\mathcal{H}}}.
\]
\end{lemma}
\begin{proof}
Let $g$ be any function in $\mathcal{H}$ such that $g(z)=a$. Then by Cauchy-Schwartz's inequality
\[
|a|=|\braket{g, k^\mathcal{H}_z}_\mathcal{H}|\leq||g||_\mathcal{H}\cdot||k^\mathcal{H}_z||_\mathcal{H},
\]
and therefore
\[
\inf\{||g||_\mathcal{H}\,\,|\,\,g\in\mathcal{H},\, g(z)=a \}\geq\frac{|a|}{||k^\mathcal{H}_z||_\mathcal{H}},
\]
since $g$ was arbitrary chosen. To prove the other inequality, it suffices to find one function $g$ in $\mathcal{H}$ such that $g(z)=a$ and $||g||_\mathcal{H}=|a|/||k^\mathcal{H}_z||_\mathcal{H}$. The function
\[
g=\frac{a}{||k^\mathcal{H}_z||_\mathcal{H}^2}k^\mathcal{H}_z
\]
has the requested properties.
\end{proof}
We can prove now Theorem $\ref{theo:dkernels}$:
\begin{proof}[Proof of Theorem \ref{theo:dkernels}]
By Lemma \ref{lemma:dist},
\begin{equation}
\label{eqn:sep:kernels}
\begin{split}
\dist\bigg(\frac{k_z}{||k_z||}\,\,,\,\,H\bigg)=&\frac{1}{||k_z||\inf\{||h||\,\,|\,\, h(z)=1\,,\,h=\Theta g\,,\,g\in\h^2\}}\\
=&\frac{1}{||k_z||\inf\{||g||\,\,|\,\, g(z)=1/\Theta(z)\}}.
\end{split}
\end{equation}
 By applying Lemma \ref{lemma:h2} to the rightmost side of \eqref{eqn:sep:kernels}, we get \eqref{eqn:dist:span} and we conclude the proof.
\end{proof}
Equation \eqref{eqn:sepker} says, intuitively, the following: the angle between any model space and the span of all the others  is large enough so that, fixed a kernel function $k_w$ in $\h^2$, the span of all the model spaces except the one which is closest to $k_w$ is at a uniform positive distance from $k_w$. To give a formal justification to our intuition, we are going now to quantify the relation between condition \eqref{eqn:sepker} and the actual sines between the model spaces involved.\\
A fundamental tool for what follows is Carleson's corona Theorem:
\begin{theo}[Carleson's corona Theorem]
\label{theo:corona}
Let $u_1$ and $u_2$ be two inner functions in $\h^2$ so that
\begin{equation}
\label{eqn:corona}
\inf_{z\in\D}\max\{|u_1(z)| , |u_2(z)|\}\geq\delta>0. 
\end{equation}
Then there exists a constant $C_\delta>0$ (depending only on $\delta$) and two functions $g_1$ and $g_2$ in $\h^\infty$ of norm less than $C_\delta$ such that
\[
u_1g_1+u_2g_2=1.
\]
\end{theo}
For instance, see \cite[Ch. VIII, Th. 2.1]{gar}.\\
In \cite[Lec. IX]{nik2}, condition \eqref{eqn:corona} is related to the sine between the model spaces of $k_1$ and $k_2$. Theorem \ref{theo:dist:model} below gives another quantitative version of such a relation.\\
Let $\Theta_1$ and $\Theta_2$ be two inner functions in $\h^2$, and let 
\[
\begin{split}
&H_1:=\h^2\ominus\Theta_1\h^2\\
&H_2:= \h^2\ominus\Theta_2\h^2
\end{split}
\]
be their model spaces. Suppose that $\Theta_1$ and $\Theta_2$ have a common zero $z_0$ whose kernel function is $k_0$. Then clearly $k_0$ belongs to both $H_1$ and $H_2$, and therefore $\sin(H_1, H_2)=0$. If $\Theta_1$ and $\Theta_2$ are Blaschke products, then we can quantify more precisely the relation between the angle between their model spaces and the separation of their zeros:
\begin{theo}
\label{theo:dist:model}
Let $B_1$ and $B_2$ be two Blaschke products, $H_1$ and $H_2$ their model spaces, respectively.
\begin{description}
\item[(i)] If 
\[
\inf_{z\in\D}\max\{|B_1(z)|, |B_2(z)|\}\leq\varepsilon<1,
\]
then for $i=1, 2$ there exists $x_i$ in $H_i$ so that $||x_i||\geq\sqrt{1-\epsilon^2}$ and $||x_1-x_2||\leq2\varepsilon$;
\item[(ii)] If 
\[
\inf_{z\in\D}\max\{|B_1(z)|, |B_2(z)|\}\geq\delta>0,
\]
then there exists a constant $c_\delta>0$ (depending only on $\delta$) such that \,$\sin(H_1, H_2)>c_\delta$.
\end{description}
\end{theo}
\begin{proof}
Let $P_i$ be the orthogonal projection onto $H_i$, $i=1, 2$, and $P$ be the orthogonal projection onto $H$, the subspace spanned by $H_1$ and $H_2$. Observe that $H=\h^2\ominus B_1B_2\h^2$, since $B_1$ and $B_2$ are Blaschke products.
\begin{description}
\item[(i)] Let $z$ be a point of $\D$ such that $|B_i(z)|\leq\varepsilon$, $i=1,2$, and let 
\[
\begin{split}
&x_i:=P_i\left(\frac{k_z}{||k_z||} \right)\qquad i=1, 2\\
&y=P\left(\frac{k_z}{||k_z||} \right).
\end{split}
\]
By Theorem \ref{theo:dkernels}, $||x_i||\geq\sqrt{1-\varepsilon^2}$ and 
\[
||x_i-y||^2=|B_i(z)|^2(1-|B_j(z)|)^2\leq|B_i(z)|^2\leq\varepsilon^2, \qquad \{i, j\}=\{1, 2 \}
\]
since $x_i=P_i(y)$, $i=1, 2$. The triangle inequality concludes the proof.
\item[(ii)] Let $f_1$ in $H_1$ have $||f_1||=1$, and $f_2=P_2(f_1)$. Therefore $f_1=f_2+g$, where $g$ belongs to $B_2\h^2$ and $\dist(f_1, H_2)=||g||$. Thanks to Theorem \ref{theo:corona} there exist $h_1$ and $h_2$, $||h_i||_\infty\leq C_\delta$, so that $h_1B_1+h_2B_2=1$. Therefore 
\[
f_1=h_1B_1f_1+h_2B_2f_1,
\]
and by orthogonality and the Cauchy-Schwartz inequality
\[
1=\braket{f_1, f_1}=\braket{h_2B_2 f_1, f_1}=\braket{h_2B_2f_1, g}\leq ||h_2||\cdot||g||\leq C_\delta\cdot\dist(f_1, H_2).
\]
To conclude the proof, it suffices to choose $c_\delta=1/C_\delta$ and observe that $f_1$ was chosen arbitrarily among all the unit vectors in $H_1$.
\end{description}
\end{proof}
As a consequence, for a sequence of matrices the notions of strong separation in Definition \ref{defi:sep} and Definition \ref{defi:sep:model} coincide:
\begin{coro}
Let $(A_n)_{n\in\N}$ be a sequence of matrices with spectra in $\D$ and let $(H_n)_{n\in\N}$ be the associated sequence of  model spaces. Then $(A_n)_{n\in\N}$ is strongly separated if and only if $(H_n)_{n\in\N}$ is strongly separated. 
\end{coro}
\section{Carleson Measures and Bessel Systems }
\label{sec:bessel}
Let $\Lambda=(\lambda_n)_{n\in\N}$ be a sequence in $\D$. One can state a separation condition for the points in $\Lambda$ by controlling the norm of the embedding of $\h^2$ into $\mathrm{L}^2(\D, \mu)$, where $\mu$ is a measure that depends on the points of $\Lambda$:
\begin{defi}
Let $\mathcal{H}$ be a reproducing kernel Hilbert space on a measure set $(X, \mathcal{A})$. A positive measure $\mu$ on $X$ is a Carleson measure for $\mathcal{H}$ if 
\[
||f||_{\mathrm{L^2}(X, \mathcal{A}, \mu)}\leq C\cdot||f||_{\mathcal{H}},\qquad f\in\mathcal{H}.
\]
The least $C$ for which the above inequality holds is the Carleson constant of $\mu$.
\end{defi}
This yields another characterization for interpolating sequences \cite{carlocm}: $\Lambda$ is interpolating if and only if  
\begin{equation}
\label{eqn:ws}
\inf_{n\ne k}\rho(\lambda_n, \lambda_k)>0
\end{equation}
and 
\begin{equation}
\label{eqn:carlo}
\sum_{n\in\N}(1-|\lambda_n|^2)\delta_{\lambda_n}
\end{equation}
is a Carleson measure for $\h^2$. Condition \eqref{eqn:ws} is called \emph{weak separation} for $\Lambda$. It comes natural to ask ourselves if such a characterization extends to sequences of matrices.\\

In \cite[Prop. 9.5]{john}, one can find a characterization for a measure like \eqref{eqn:carlo} to be Carleson. Let $X\subset\C$ be a set, and $\mathcal{H}$ be a reproducing kernel Hilbert space on $X$:
\begin{prop}
\label{prop:john}
Let $(\lambda_n)_{n\in\N}$ be a sequence in $X$, and $k^\mathcal{H}_n$ be the reproducing kernel at $\lambda_n$, for any $n$. The following are equivalent:
\begin{description}
\item[(i)] The measure $\sum_{n\in\N}||k^\mathcal{H}_n||_\mathcal{H}^{-2}\delta_{\lambda_n}$ is Carleson;
\item[(ii)] The sequence $(k^\mathcal{H}_n/||k^\mathcal{H}_n||_\mathcal{H})_{n\in\N}$ of normalized kernels is a Bessel system in $\mathcal{H}$.
\end{description}
\end{prop} 
Recall that a \emph{Bessel system} in a Hilbert space $\mathcal{H}$ is a sequence $(x_n)_{n\in\N}$ so that the linear map $T_E\colon\mathcal{H}\to\overline{\spa}\{x_n\}_{n\in\N}$ that extends
\begin{equation}
\label{eqn:syst}
T_E(e_n)=x_n
\end{equation}
is bounded, for any $E=(e_n)_{n\in\N}$ orthonormal basis of $\mathcal{H}$. If \eqref{eqn:syst} extends to a bounded and invertible linear map, we will say that $(x_n)_{n\in\N}$ is a \emph{Riesz system}.\\
These almost-orthogonality notions can be extended to sequences $(X_n)_{n\in\N}$ of closed subspaces of a given Hilbert space:
\begin{defi}
Let $X=(X_n)_{n\in\N}$ be a sequence of closed subspaces of a Hilbert space. Then
\begin{description}
\item[(a)] $X$ is a Bessel system if there exists a positive $C>0$ such that, for any $(x_n)_{n\in\N}$ in $\mathcal{H}$ so that $x_n$ belongs to $X_n$, the map $T$ in equation \eqref{eqn:syst} extends to a linear operator bounded by $C$;
\item[(b)] $X$ is a Riesz System if there exists a positive $C>0$ such that, for any $(x_n)_{n\in\N}$ in $\mathcal{H}$ so that $x_n$ belongs to $X_n$, the map $T$ in equation \eqref{eqn:syst} extends to a linear operator bounded by $C$ and bounded below by $1/C$.
\end{description}
\end{defi}
A sequence of matrices $(A_n)_{n\in\N}$ is interpolating if and only if the sequence $(H_n)_{n\in\N}$ of their model spaces is a Riesz system. Indeed, \cite[Th. 3.2.14.]{nik}, condition (i) and (ii) of Theorem \ref{theo:nik} are equivalent to 
\begin{equation}
\label{eqn:model:riesz}
(\h^2\ominus\Theta_n\h^2)_{n\in\N}\quad\text{is a Riesz system in }\,\h^2.\tag{$\star$}
\end{equation}
Observe that if our matrices are just scalars, say $(\lambda_n)_{n\in\N}$, then their model spaces are just one dimensional subspaces of $\h^2$ generated by their kernel functions. In this case, weak separation corresponds to weak separation of model spaces:
\begin{defi}
A sequence $(K_n)_{n\in\N}$ of closed subspaces of a Hilbert space is weakly separated if 
\[
\inf_{j\ne n}\sin(K_n, K_j)>0.
\]
\end{defi}
Thanks to Theorem \ref{theo:dist:model}, the lines generated by the kernel functions at the points $(\lambda_n)_{n\in\N}$ are weakly separated if and only if there exists a positive $\delta$ so that, for any $z$ in $\D$,
\[
\inf_{n\ne j}\max\{|b_{\lambda_n}(z)|, |b_{\lambda_j}(z)|\}\geq\delta.
\]
This is equivalent to saying that $(\lambda_n)_{n\in\N}$ is weakly separated: for instance, see \cite[Lemma 3.2.18]{nik}. Thus Carleson's Theorem and Proposition \ref{prop:john} give
\begin{theo}
\label{theo:syst}
Let $(k_n)_{n\in\N}$ be a sequence of normalized kernel functions in $\h^2$, and let 
\[
K_n=\spa\{k_n\}.
\]
Then $(K_n)_{n\in\N}$ is a Riesz system if and only if it is a weakly separated Bessel system.
\end{theo}
It is important to observe is that Theorem \ref{theo:syst} does not extend to any sequence $(x_n)_{n\in\N}$ in a Hilbert space $\mathcal{H}$. Indeed, let $(e_n)_{n\in\N}$ be an orthonormal basis of $\mathcal{H}$, and let $\gamma_n$ be a sequence that converges to $0$. Then
\begin{equation}
\label{eqn:sys:count}
x_n:=\begin{cases}
e_{n-2}+e_{n-1}+\gamma_ne_n\quad&\text{if }\, n\,\,\text{is a multiple of }\,3\\
e_n\quad&\text{otherwise}
\end{cases}
\end{equation}
 is a weakly separated Bessel system, but it is not a Riesz system, as 
\[
\dist(e_{3n}, \spa\{e_{3n-2}, e_{3n-1}\})=\gamma_n\to0,
\]
as $n$ goes to infinity. It turns out then that Szegö kernels are somehow special: a sequence like \eqref{eqn:sys:count} can not be made by kernel functions since, thanks to Theorem \ref{theo:dkernels}, a normalized kernel function approaches the span of two others if and only if it actually approaches one of them. Indeed, for any $w_1$, $w_2$ and $w_3$ in $\D$,
\[
\begin{split}
&\dist\left(\frac{k_{w_3}}{||k_{w_3}||}\,,\,\spa\{k_{w_1}, k_{w_2}\}\right)=|b_{w_1}(w_3)b_{w_2}(w_3)|\\
=&\dist\left(\frac{k_{w_3}}{||k_{w_3}||}\,,\,\spa\{k_{w_1}\}\right)~\dist\left(\frac{k_{w_3}}{||k_{w_3}||}\,,\,\spa\{ k_{w_2}\}\right).
\end{split}
\]

We asked ourselves if Theorem \ref{theo:syst} can be generalized to a sequence of model spaces associated to a sequence of matrices $A=(A_n)_{n\in\N}$. By \eqref{eqn:model:riesz} and Theorem \ref{theo:dist:model}, a Riesz system of model spaces is strongly separated, and therefore a weakly separated Bessel system. The converse also holds, if the sizes of the matrices in $A$ are bounded:
\begin{theo}
\label{theo:matbessel}
Let $A=(A_n)_{n\in\N}$ be a sequence of matrices with spectra in the unit disc, and assume that their dimensions are uniformly bounded. Then $A$ is interpolating if and only if the sequence of its model spaces is a weakly separated Bessel system.
\end{theo}
To prove Theorem \ref{theo:matbessel}, we will invoke the \emph{Feichtinger's conjecture}, which asserts that any Bessel sequence $(x_n)_{n\in\N}$ of vectors in a Hilbert space is a finite disjoint union of Riesz systems. It has been shown, \cite{casazza} \cite{weaver}, that the Feichtinger conjecture is equivalent to many other conjecture in operator theory, including the Paving conjecture, who had been proved by the celebrated work of Marcus, Spielmann and Srivastava \cite{marcus}. 
\begin{proof}[Proof of Theorem \ref{theo:matbessel}]
Let $(H_n)_{n\in\N}$ be the sequence of model spaces associated to $A$. We will show that if $(H_n)_{n\in\N}$ is a weakly separated Bessel system, then $A$ is strongly separated. Let $k$ be the maximal number of distinct eigenvalues of a matrix in $A$, and $m$ the maximal order of any eigenvalue of the matrices in $A$. Since the sizes of $A$ are uniformly bounded,  both $k$ and $m$ are finite, and then it suffices to take care of the case in which the Blaschke factors of the matrices in $A$ have the form
\[
B_{A_n}=\prod_{j=1}^kb_{\lambda_{n, j}}^m.
\]
 Since the model spaces of $A$ are finite dimensional and a Bessel system, the sequence
\[
\{k_{\lambda_{n, j}}\,\,| \,\, j=1,\dots, k,\,\, n\in\N\}
\]
is a Bessel sequence, and therefore it is the finite union $\bigcup_{l=1}^r\Lambda_l$ of Riesz sequences. Fix $z$ in $\D$ and let $n$ be the index that minimizes $|B_{A_n}(z)|$. We want to show that there exists some choice of a positive $\delta$ independent of $z$ so that
\[
\prod_{k\ne n}|B_{A_k}(z)|\geq\delta.
\] 
Let $\lambda_l$ be the point of $\Lambda_l$ so that
\[
\rho(z, \lambda_l)=\min_{\lambda\in\Lambda_l}\rho(z, \lambda).
\]
Thanks to \eqref{eqn:model:riesz}, for kernel functions strong separation is equivalent to being a Riesz system, and therefore
\[
\prod_{\lambda\in\Lambda_{l}\setminus\{\lambda_l\}}|b_{\lambda}(z)|\geq\delta',\qquad l=1,\dots, r,
\]
where $\delta'$ does not depend on $z$. If  $\Tilde{\Lambda}=\{k_{\lambda_l},\,\,|\,\,l=1, \dots, r\}$  is contained $H_n$, then we conclude by setting $\delta=\delta'^{rm}$. If $s$ elements in $\tilde{\Lambda}$ are not in $H_n$, weak separation implies that their distance from $z$ is bounded below by a positive $\delta''$, independently of $z$. Then we can conclude by setting $\delta=\delta'^{(r-s)m}\delta''^{sm}$.
\end{proof}
Unfortunately, we were not able to drop the hypothesis on the sizes of the matrices in $A$ in Theorem \ref{theo:matbessel}. To do so, we would need a more general statement of the Feichtinger conjecture to be true for a sequence of model spaces:
\begin{question}
Is any Bessel system of model spaces a finite union of Riesz systems?
\end{question}
A positive answer for model spaces associated to convergent Blaschke factors would extend Theorem \ref{theo:matbessel} to any sequence of matrices.\\
 Indeed, let $(B_n)_{n\in\N}$ be the sequence of Blaschke factors associated to a sequence of matrices with spectra in the unit disc, and assume that the associated sequence $(H_n)_{n\in\N}$ of model spaces  is weakly separated and Bessel. Suppose that $\N$ can be partitioned by $\bigcup_{i=1}^M\sigma_i$ and any $\mathcal{X}_i=(H_n)_{n\in\sigma_i}$ is a Riesz system. Let
\[
\delta=\inf_{z\in\D}\inf_{k\ne n\in\N}\max\{|B_k(z)|, |B_n(z)|\}.
\]
Then $\delta>0$, thanks to weak separation and Theorem \ref{theo:dist:model}.\\
 Fix $z$ in $\D$, and let $j_i$ be, for any $i=1,\dots,M$, so that
\[
|B_{j_i}(z)|=\min_{n\in\sigma_i}|B_n(z)|.
\]
By Theorem \ref{theo:nik} and \eqref{eqn:model:riesz}
\[
\prod_{n\in\sigma_i\setminus\{j_i\}}|B_n(z)|\geq\delta_i\qquad i=1,\dots,M,
\]
for some positive choice of $\delta_1,\dots,\delta_M$ independent on $z$. Observe that the index $j$ in $\N$ that minimizes $|B_n(z)|$ is in $(j_i)_{i=1}^M$. Without loss of generality let $j=j_1$: therefore
\[
\max_{n\in\N}\prod_{k\ne n}|B_k(z)|=\left(\prod_{k\ne j_1,\dots,j_M}|B_k(z)|\right)\left(\prod_{i=2}^M |B_{j_i}(z)|\right)\geq\delta^{M-1}\prod_{i=1}^M\delta_i>0,
\]
independently of $z$. By Theorem \ref{theo:nik} and \eqref{eqn:model:riesz}, $(H_n)_{n\in\N}$ is a Riesz system and the associated sequence of matrices is interpolating.

\end{document}